\documentclass{amsart}
\usepackage{amssymb}
\usepackage{color}
\usepackage{float}
\usepackage[all,cmtip]{xy}
\usepackage{graphicx}

\newcommand{\Acal}{{\mathcal{A}}}

\newcommand{\Ecal}{{\mathcal{E}}}

\newcommand{\FF}{\mathbb{F}}

\newcommand{\NN}{\mathbb{N}}

\newcommand{\PP}{\mathbb{P}}
\newcommand{\QQ}{\mathbb{Q}}
\newcommand{\RR}{\mathbb{R}}

\newcommand{\ZZ}{\mathbb{Z}}

\newcommand{\id}{\textbf{\textit{I}}}

\newcommand{\Cont}{\operatorname{Cont}}

\newcommand{\Vol}{\operatorname{Vol}}
\newcommand{\RFH}{\operatorname{RFH}}
\newcommand{\RFC}{\operatorname{RFC}}
\newcommand{\HM}{\operatorname{HM}}

\renewcommand{\id}{\operatorname{id}}

\newtheorem{theorem}{Theorem}
\newtheorem*{theorem*}{Theorem}
\newtheorem{proposition}{Proposition}[section]
\newtheorem{lemma}[proposition]{Lemma}

\newtheorem{corollary}[theorem]{Corollary}

\theoremstyle{definition}

\newtheorem{example}[proposition]{Example}

\theoremstyle{remark}
\newtheorem{remark}[proposition]{Remark}

\numberwithin{equation}{section}

\begin{document}

\title{Lower complexity bounds for positive contactomorphisms}

\author{Lucas Dahinden}
\address{Universit\'e de Neuch\^atel (UNINE)}
\curraddr{Institut de Math\'ematiques, Rue Emile-Argand 2, 2000 Neuch\^atel}
\email{l.dahinden@gmail.com}
\thanks{The first author was supported by SNF grant 200021-163419/1.}


\date{June 15, 2015.}

\begin{abstract}
	Let $S^*Q$ be the spherization of a closed connected manifold of dimension at least two. Consider a contactomorphism $\varphi$ that can be reached by a contact isotopy that is everywhere positively transverse to the contact structure. In other words, $\varphi$ is the time-1-map of a time-dependent Reeb flow. We show that the volume growth of $\varphi$ is bounded from below by the topological complexity of the loop space of $Q$. Denote by $\Omega Q_0(q)$ the component of the based loop space that contains the constant loop.
\begin{theorem*}
	If the fundamental group of $Q$ or the homology of $\Omega Q_0(q)$ grows exponentially, then the volume growth of $\varphi$ is exponential, and thus its topological entropy is positive.
\end{theorem*}
A similar statement holds for polynomial growths. This result generalizes work of Dinaburg, Gromov, Paternain and Petean on geodesic flows and of Macarini, Frauenfelder, Labrousse and Schlenk on Reeb flows.\\
\indent Our main tool is a version of Rabinowitz--Floer homology developed by Albers and Frauenfelder.
\end{abstract}

\maketitle

\section{Introduction and result}\label{sec:intro}
Before stating the main theorem, we define positive contactomorphisms, the volume growth of maps and the growth of loop spaces. All our objects are in the smooth category and all manifolds have dimension $\geq2$.\\
\\
The spherization $S^*Q$ of a manifold $Q$ is the space of positive line elements in the cotangent bundle $T^*Q$. The tautological one-form $\lambda$ on $T^*Q$ does not restrict to $S^*Q$, but its kernel does. This endows $S^*Q$ with a co-oriented contact structure~$\xi$. Choose a contact form $\alpha$ for $\xi$. We call a smooth path of contactomorphisms $\varphi^t:[0,1]\to \Cont(S^*Q)$ starting at the identity \textit{positive} if its generating vector field $X^t(\varphi^t(x)):=\frac d{dt}\varphi^t(x)$ is positively transverse to the contact structure: $\alpha(X^t)>0$. A contactomorphism $\varphi$ is called \textit{positive} if there is a positive path $\varphi^t$ of contactomorphisms with $\varphi^0=id,\:\varphi^1=\varphi$. This notion is independent of the choice of contact form. \\

Let $\varphi:M\to M$ be a smooth diffeomorphism of a manifold $M$. Let $S\subset M$ be a compact submanifold and fix a Riemannian metric $g$ on $M$. We denote by $\gamma_{\rm{vol,pol}}(\varphi;S)$ and $\gamma_{\rm{\rm{vol,exp}}}(\varphi;S)$ the polynomial and exponential volume growth of~$S$ under iterations of $\varphi$, where the volume is induced by $g$:
\begin{eqnarray*}
	\gamma_{\rm{vol,pol}}(\varphi;S)&=&\liminf_{m\to\infty}\frac 1{\log m}\log\Vol(\varphi^m(S)),\\
	\gamma_{\rm{\rm{vol,exp}}}(\varphi;S)&=&\liminf_{m\to\infty}\frac 1{ m}\log\Vol(\varphi^m(S)).	
\end{eqnarray*}
The polynomial and exponential volume growth of a map is the supremum of polynomial and exponential volume growths over all compact submanifolds:
\begin{eqnarray*}
	\gamma_{\rm{vol,pol}}(\varphi)&=&\sup_S\gamma_{\rm{vol,pol}}(\varphi;S),\\
	\gamma_{\rm{vol,exp}}(\varphi)&=&\sup_S\gamma_{\rm{vol,exp}}(\varphi;S).
\end{eqnarray*}
These numbers are clearly independent of the choice of Riemannian metric.\\

Let $Q$ be a connected manifold and consider the connected component $\Omega Q_0(q)$ of contractible loops of the space of loops based at $q$. Denote by $\PP$ the set of primes and zero. For $p$ prime let $\FF_p$ be the field $\ZZ/p\ZZ$ and $\FF_0=\QQ$. Define the homological polynomial and exponential growth of $\Omega Q_0(q)$ by
\begin{eqnarray*}
	\gamma_{\rm{pol}}(\Omega Q_0(q))&:=&\sup_{p\in\PP}\liminf_{m\to\infty}\frac 1{\log m}\log\sum_{k=0}^m\dim(H_k(\Omega Q_0(q);\FF_p)),\\
	\gamma_{\rm{exp}}(\Omega Q_0(q))&:=&\sup_{p\in\PP}\liminf_{m\to\infty}\frac 1{ m}\log\sum_{k=0}^m\dim(H_k(\Omega Q_0(q);\FF_p)).	
\end{eqnarray*}
Note that $\Omega Q_0(q)$ is homotopy equivalent to any connected component of the space $\Omega Q(q,q')$ of paths in $Q$ from $q$ to $q'$. Thus if we had used in the above definition a connected component of $\Omega Q(q,q')$ instead of $\Omega Q_0(q)$, we would have got the same number.\\

Denote by $\gamma_{\rm{pol}}(\pi_1(Q))$ and $\gamma_{\rm{exp}}(\pi_1(Q))$ the polynomial and exponential growth of the fundamental group of $Q$ for some (and thus every) set of generators. With this notation we can state the main result of this paper.

\begin{theorem}\label{growththeorem}
Let $\varphi$ be a positive contactomorphism on the spherization $S^*Q$ of the closed manifold $Q$ and let $q$ be any point in $Q$.
\begin{enumerate}
	\item If $\gamma_{\rm exp}(\pi_1(Q))>0$ or if $\gamma_{\rm exp}(\Omega Q_0(q))>0$, then
	$$\gamma_{\rm{vol,exp}}(\varphi)\ \geq\ \gamma_{\rm{vol,exp}}(\varphi;S_q^*Q)\:>\:0.$$
	\item If $\gamma_{\rm pol}(\pi_1(Q))$ and $\gamma_{\rm pol}(\Omega Q_0(q))$ are finite, then $$\gamma_{\rm{vol,pol}}(\varphi)\ \geq\ \gamma_{\rm{vol,pol}}(\varphi;S_q^*Q)\ \geq\  \gamma_{\rm{pol}}(\pi_1(Q))+\gamma_{\rm{pol}}(\Omega Q_0(q))-1.$$
\end{enumerate}
\end{theorem}
Yomdin and Newhouse~\cite{Y,N} related the exponential volume growth to the topological entropy. They showed that $\gamma_{\rm{vol,exp}}(\varphi)=h_{\rm{top}}(\varphi)$. This results in the following reformulation of Theorem~\ref{growththeorem}~(1).
\begin{corollary}
If $\pi_1(Q)$ or $\Omega Q_0(q,q)$ grows exponentially, then for every positive contactomorphism $\varphi$ on the spherization $S^*Q$,
$$h_{top}(\varphi)>0.$$
\end{corollary}

The first versions of this theorem were proved by Dinaburg, Gromov, Paternain and Petean for geodesic flows, using Morse theory~\cite{Dina, Gro78, Pat97, PatPet03, PatPet06}, see also~\cite{FraSch14}. Frauenfelder--Schlenk~\cite{FS:GAFA} generalized the theorem to certain Hamiltonian flows on $T^*M$, using Lagrangian Floer homology. A further generalization to Reeb flows was found by Macarini--Schlenk \cite{MacSch11} (exponential) and Frauenfelder--Labrousse--Schlenk \cite{FLS} (polynomial), also using Lagrangian Floer homology. In this paper we extend these results to positive contactomorphisms. These maps can be realized as time-dependent Reeb flows. 

\begin{remark}
There are several approaches to dealing with the time-dependence of the Reeb flow. 
One is to absorb the time-dependence in an additional space factor, such as $T^*S^1$.
Another approach is to cook up an action functional for our problem,
and to deform it to the action functional for a Reeb flow.
While we did not succeed with the geometric approach, the second approach worked out well.
\end{remark}

\begin{remark} Positive entropy for all Reeb flows on many contact 3-manifolds different from spherizations has recently been established by Alves in~\cite{A1, A2, A3}.
\end{remark}

The identity map is a non-negative contactomorphism that can be uniformly approximated by positive contactomorphisms by slowing down a fixed positive contact isotopy. The class of positive contactomorphisms thus seems to be the largest natural class of contact geometric maps for which one has positive topological entropy on $S^*Q$ under the topological condition on $Q$ given in Theorem~\ref{growththeorem}. Indeed, without the positivity assumption various scenarii are possible:
\begin{example}\label{prop:examples}
	\begin{enumerate}
		\item There are closed manifolds $Q$ with $\pi_1(Q)$ of exponential growth whose spherization $S^*Q$ carries a non-negative contactomorphism $\varphi$ such that $h_{\rm{top}}(\varphi)=0$ and such that $\varphi$ is generated by an autonomous contact isotopy that is positive outside a submanifold of positive codimension.
		\item There are closed manifolds $Q$ with $\pi_1(Q)$ of exponential growth whose spherization $S^*Q$ carries a non-negative contactomorphism $\varphi$ such that $h_{\rm{top}}(\varphi)>0$ and such that $\varphi$ is generated by an autonomous contact isotopy that restricts to the identity on a subset of $S^*Q$ with nonempty interior.
	\end{enumerate}
\end{example}
Two specific examples are given in Section~\ref{examples}.

\subsection*{Acknowledgements}
	I wish to thank Peter Albers, Marcelo Alves, Urs Frauenfelder, Felix Schlenk and the anonymous referee for their valuable suggestions. I am particularly grateful to Urs for suggesting to use this version of Rabinowitz--Floer homology. This work is supported by SNF grant 200021-163419/1.

\section{Recollections}\label{sec:tce}

In this section we first represent the spherization $S^*Q$ as a hypersurface in $T^*Q$ and specify the choice of positive path of contactomorphism. Then we cite a theorem that relates the volume growth of a flow on a spherization to the volume growth of an extension of the flow to the sublevel. Finally we state some facts about the version of Rabinowitz--Floer homology that will be used in the proof of Theorem~\ref{growththeorem}. The precise definition of this Rabinowitz--Floer homology is given in Section~\ref{sec:appendix}.\\
 
	The spherization of a manifold can be naturally represented as a fiberwise starshaped hypersurface $\Sigma\subset T^*Q$ in the cotangent bundle with contact structure $\xi=\ker\lambda|_\Sigma$, where $\lambda$ is the Liouville one-form. The map that sends a positive line element to its intersection with~$\Sigma$ is a contactormorphism. The radial dilation of a fiberwise starshaped hypersurface by a positive function is a contactomorphism to its image. Every contact form of $(S^*Q,\xi)$ is realized as $\lambda|_\Sigma$ for some hypersurface. 
	The symplectization $(\Sigma\times \RR_{>0},d(r\alpha))$ naturally embeds into $T^*Q\backslash Q$. A contact isotopy $\varphi^t$ admits a lift to a Hamiltonian isotopy of $\Sigma\times\RR_{>0}$, generated by a time-dependent 1-homogenous Hamiltonian $H^t$. \\
	
 Fix $(\Sigma\subset T^*Q,\lambda|_\Sigma)$ representing $S^*Q$. For every positive contactomorphism $\varphi$ the path $\{\varphi^t\}_{t\in[0,1]}$ can be chosen to be the Reeb flow for $t$ near $0$ and $1$, as explained in the second part of the proof of~\cite[Proposition~6.2]{FLS}. This means that the contact Hamiltonian $h^t$ generating $\varphi^t$ is constant $\equiv c_0$ for $t$ near $0$ and $1$. Thus $h^t$ permits smooth periodic or constant extensions. In this paper we always extend $\varphi^t$ such that $h^t$ is constant $c_0$ for $t\leq0$ and periodic for $t\geq0$, see Figure~\ref{fig:graph1}. The reason for this will become clear in the proof of Theorem~\ref{continuation}.

\begin{figure}[h]
	\centering
		\includegraphics{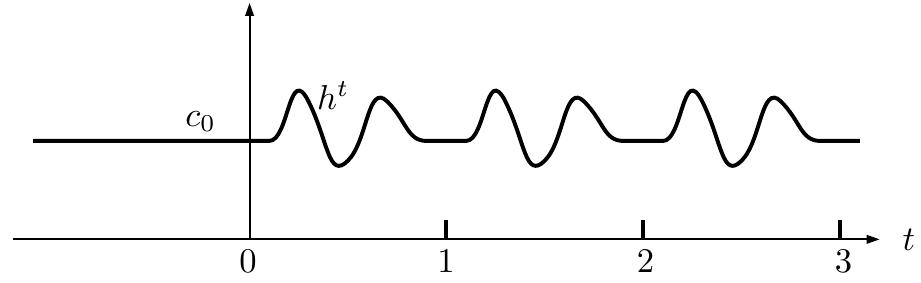}
	\caption{The function $h^t$, extended to $\RR$}.
	\label{fig:graph1}
\end{figure}

Fix a Riemannian metric on $Q$ and consider the induced metrics on $TQ$ and $T^*Q$. Denote by $\mu_k$ the induced volume form on $k$-dimensional submanifolds of $T^*Q$ and denote by $\Vol(\cdot)$ the integral of $\mu_k$ on $\cdot$. Using this metric, represent $S^*Q$ as the 1-cosphere-bundle in $T^*Q$. Let $\varphi$ be a positive contactomorphism on $S^*Q$ and choose a path of positive contactomorphisms $\varphi^t$ with $\varphi^1=\varphi$. For each $t$ we extend $\varphi^t$ to $T^*Q\backslash Q$ by 
\begin{eqnarray}\label{extension}\varphi^t(q,s p)&=&s\varphi^{s t}(q,p)\end{eqnarray} 
for $s>0$. Note that the $q$-coordinate of $\varphi^t(q,s p)$ is the $q$-coordinate of $\varphi^{s t}(q,p)$. Also note that the extension~(\ref{extension}) is not the Hamiltonian lift mentioned above. The following theorem relates the volume growth of a sphere $S^*_qQ$ with the volume growth of its punctured sublevel disk $\dot D^*_qQ$ under a general twisted periodic flow. The proof can be found in the proof of Proposition~4.3 in~\cite{FLS}, where the statement is proven for the slow growth of Reeb flows. 
\begin{theorem}\label{sublevel}
		Let $\varphi^t:S^*Q\to S^*Q$ be a smooth family of diffeomorphisms with $\varphi^0=\id$ whose generating vector field is 1-periodic. Extend $\varphi^t$ to $T^*Q\backslash Q$ by~(\ref{extension}). Then
\begin{eqnarray*}
	\gamma_{\rm{vol,exp}}(\varphi^1;\dot D^*_qQ)&\leq& \gamma_{\rm{vol,exp}}(\varphi^1;S^*_qQ),\\
	\gamma_{\rm{vol,pol}}(\varphi^1;\dot D^*_qQ)-1&\leq& \gamma_{\rm{vol,pol}}(\varphi^1;S^*_qQ).
\end{eqnarray*}
\end{theorem}
\begin{remark}
	The proof in~\cite{FLS} extends the flow on a hypersurface to its sublevel by extending the contact Hamiltonian to a homogeneous symplectic Hamiltonian. If one extends the flow directly as in~(\ref{extension}) without using Hamiltonians and on a 1-cosphere bundle for an arbitrary Riemannian metric, the proof still goes through. 
\end{remark}

Albers and Frauenfelder~\cite{AF1} built a version of Rabinowitz--Floer homology for the space $\Omega^1_{q,q'}$ of $W^{1,2}$ paths from $T^*_{q}Q$ to $T^*_{q'}Q$. Given a positive path of contactomorphisms~$\varphi^t$ we construct a certain modification~$H^t$ of a 1-homogenous Hamiltonian in $T^*Q$ corresponding to $\varphi^t$, for details see Section~\ref{sec:appendix}. Define the functional $\Acal(\varphi^t;q,q'):\Omega^1_{q,q'}\times\RR\to\RR$ by
\begin{eqnarray}
\label{functional}\Acal(\varphi^t;q,q')(x,\eta)&=&\int_0^1x^*\lambda-\eta\int_0^1H^{\eta t}(x(t))\;dt.
\end{eqnarray}
A pair $(x,\eta)\in\Omega_{q,q'}^1\times\RR$ is a critical point of the functional~(\ref{functional}) if and only if it satisfies the equations
$$\begin{cases}
		\dot x(t)&=\:\: \eta X_{H^{\eta t}}(x(t)),\\
		H^{\eta}(x(1))&=\:\: 0.
\end{cases}$$
The first equation implies that $x$ is an orbit of $X_{H^t}$, but with time scaled by $\eta$. The second equation implies that the orbit ends on $(H^{\eta})^{-1}(0)$. For $H^t=H$ autonomous $H^{-1}(0)$ is a hypersurface for which $\eta$ plays the role of a Lagrange multiplier. For time-dependent $H^t$, however, there is no such surface and $H^{\eta t}(x(t))$ might be very large for $t<1$. The chain complex $\RFC^T=\RFC^T(\varphi^t;q,q')$ of the filtered Rabinowitz--Floer homology is generated by the critical points of $\Acal(\varphi^t;q,q')$ with action value $\leq T\in\RR\cup\{\infty\}$, for more details see Section~\ref{sec:appendix}. The boundary operator $\partial^T$ is defined by counting solutions of a negative gradient flow with respect to a suitable $L^2$-metric. For $T\leq T'$ denote by $\iota^{T,T'}:\RFC^T\to\RFC^{T'}$ the inclusion. We denote by $\RFC_+^T=\RFC^T/\iota^{0,T}(\RFC^0)$ the positive part of $\RFC^T$ and set $\RFH_+^T=H(\RFC_+^T,\partial_+^T)$, where $\partial_+^T:\RFC_+^T\to\RFC_+^T$ is the induced boundary operator. For $T\leq T'$ let $\iota_{+}^{T,T'}:\RFC_+^T\to\RFC_{+}^{T'}$ be the homomorphism induced by inclusions.
 \\

The next four theorems describe the properties of $\RFH_+^T$ used in this paper. 

\begin{theorem}\label{morse}
	The functional $\Acal(\varphi^t;q,q')$ is Morse for generic points $q'\in Q$. In this case the Rabinowitz--Floer homology $\RFH_+^T(\varphi^t;q,q')$ is well-defined for all $T$. 
\end{theorem}

This theorem follows from standard theory, see~\cite[Sections 6 and 7]{AF1}. Denote by $Q_{\rm gen}$ the set of $q'$ for which $\Acal(\varphi^t;q,q')$ is Morse and that are different from $q$. Then $Q_{\rm gen}$ has full measure in $Q$. The following theorem is the key ingredient of our proof.

\begin{theorem}\label{continuation}
	Let $\varphi_i$, $i=0,1$, be two positive contactomorphisms of $S^*Q$, let $\varphi_i^t$ be corresponding positive paths of contactomorphisms and let $Q^i_{\rm gen}$ be the corresponding sets from Theorem~\ref{morse}. Then for every $q'\in Q^0_{\rm gen}\cap Q^1_{\rm gen}$ the (polynomial and exponential) growth of $\dim(\iota_{+}^{T,\infty})_*\RFH_+^T(\varphi^t_0;q,q')$ is the same as the (polynomial and exponential) growth of $\dim(\iota_{+}^{T,\infty})_*\RFH_+^T(\varphi^t_1;q,q')$.
\end{theorem}	

Thus the growth of the homology is preserved by a deformation of the flow $\varphi^t$. Theorem~\ref{continuation} is stated in~\cite[Section 7]{AF1}. We give a proof in Section~\ref{sec:appendix}.
	
\begin{theorem}\label{merry}
	Assume that $\varphi^t_g$ is a geodesic flow. Then for $q'\in Q_{\rm gen}$ the Rabinowitz--Floer homology is isomorphic to the Morse homology of the energy functional $\mathcal E(x)=\int_0^1 g(\dot x,\dot x)\;dt$ on the space $\Omega Q(q,q')$ of paths in $Q$ from $q$ to $q'$:
	\begin{eqnarray*}
		\RFH^{T}_{+}(\varphi^{t}_g;T_q^*Q,T_{q'}^*Q)&\cong&\HM^T(\mathcal E; \Omega Q(q,q')).
	\end{eqnarray*}
	This isomorphism commutes with $(\iota_{+}^{T,\infty})_*$.
\end{theorem}

This theorem is contained in Merry's work~\cite[Theorem~3.16]{M}. Note that for autonomous flows (in particular geodesic flows) the action functional is the more classical Rabinowitz--Floer action functional, which Merry used in his work:
$$\Acal(x,\eta)=\int_0^1x^*\lambda-\eta\int_0^1H^{\eta t}=\int_0^1x^*\lambda-\eta\int_0^1H.$$ 

We finally need a link between the sublevel growth of the homology of $\Ecal$ and the growth of the homology of the based loop space.
\begin{theorem}\label{theoremheap}
	Let $q'$ be non-conjugate to $q$.\\
	If $\gamma_{\rm exp}(\Omega Q_0(q))>0$ or $\gamma_{\rm exp}(\pi_1(Q))>0$, then 
	\begin{eqnarray*}
		\liminf_{T\to\infty}\frac 1T\log\dim (\iota_{+}^{T,\infty})_*\HM^T(\mathcal E; \Omega Q(q,q'))&>&0.
	\end{eqnarray*}
	If both $\gamma_{\rm pol}(\Omega Q_0(q))$ and $\gamma_{\rm pol}(\pi_1(Q))$ are finite, then 
	\begin{eqnarray*}
		\liminf_{T\to\infty}\frac 1{\log T}\log\dim (\iota_{+}^{T,\infty})_*\HM^T(\mathcal E; \Omega Q(q,q'))&\geq& \gamma_{\rm pol}(\pi_1(Q)) + \gamma_{\rm pol}(\Omega Q_0(q)).
	\end{eqnarray*}
\end{theorem}

	This is a result for geodesic flows taken from~\cite{Gro78},~\cite{P} and~\cite{Pat97} in the exponential case and from~\cite{FraSch06} in the polynomial case. 

\section{Proof of Theorem~\ref{growththeorem}}\label{sec:proof}

	Fix a Riemannian metric $g$ on~$Q$, consider the induced Riemannian metric on~$T^*Q$ and represent $S^*Q$ as the 1-cosphere-bundle in $T^*Q$ with respect to this metric as in Section~\ref{sec:tce}. Given the positive contactomorphism $\varphi:S^*Q\to S^*Q$, choose a positive path of contactomorphisms $\varphi^t$ with $\varphi^0=id,\:\varphi^1=\varphi$, extended in time as in Section~\ref{sec:tce}. Fix $q\in Q$. The exponential (polynomial) volume growth of $\varphi$ is not less than the exponential (polynomial) volume growth of the cosphere $S^*_qQ$ under $\varphi$: $$\gamma_{\rm{vol,exp}}(\varphi)\geq\gamma_{\rm{vol, exp}}(\varphi;S^*_qQ),\quad \gamma_{\rm{vol,pol}}(\varphi)\geq\gamma_{\rm{vol, pol}}(\varphi;S^*_qQ).$$ 
	Extend $\varphi^t$ to $T^*Q\backslash Q$ by~(\ref{extension}). By Theorem~\ref{sublevel}, $$\gamma_{\rm{vol,exp}}(\varphi;S^*_qQ)\geq\gamma_{\rm{vol, exp}}(\varphi^1;\dot D^*_qQ),\quad \gamma_{\rm{vol,pol}}(\varphi;S^*_qQ)\geq\gamma_{\rm{vol, pol}}(\varphi^1;\dot D^*_qQ)-1.$$
	
	The projection of $\varphi^T(\dot D^*_qQ)$ to the base manifold $Q$ has an open and dense set of regular values $Q_{\rm gen}(T)$. The set $Q^0_{\rm gen}:=\bigcap_{T\in\NN} Q_{\rm gen}(T)\backslash\{q\}$ is comeager and thus has full measure. For $q'\in Q^0_{\rm gen}$ the disk $\varphi^{T}(\dot D^*_qQ)$ intersects $\dot D^*_{q'}Q$ transversally for all $T$ and consequentially the number $\#(\varphi^T(\dot D^*_qQ)\cap \dot D^*_{q'}Q)$ is well defined. The volume of $\varphi^T(\dot D^*_qQ)$ for $T>0$ is not less than the volume of its projection (counted with multiplicity) to $Q^0_{\rm gen}$, 
\begin{eqnarray*}
	\Vol\left(\varphi^T(\dot D^*_{q}Q)\right)&\geq&\int_{Q^0_{\rm gen}}\#\left(\varphi^T(\dot D^*_qQ)\cap \dot D^*_{q'}Q\right) dq'.
\end{eqnarray*}
By the homogeneity of the flow, the elements of $\varphi^T(\dot D^*_{q}Q)\cap \dot D^*_{q'}Q$ correspond to orbits of $\varphi^t$ from $S^*_{q}Q$ to $S^*_{q'}Q$ that arrive at the latest at time $T$. \\

To count these orbits, define the Rabinowitz--Floer action functional $\Acal$ as in~(\ref{functional}). It is Morse for $q'\in Q^1_{\rm gen}$, see Theorem~\ref{morse}. The intersection $Q^0_{\rm gen}\cap Q^1_{\rm gen}$ is also comeager. Since $q\neq q'$, the critical orbits $(x,\eta)$ of $\Acal$ with $0<\Acal(x,\eta)\leq T$ are in bijection with the orbits of $\varphi^t$ from $S^*_qQ$ to $S^*_{q'}Q$ that arrive at time $\leq T$. On the other hand, these critical orbits generate the Rabinowitz--Floer chain complex $\RFC_+^T(\varphi^t;T_q^*Q,T_{q'}^*Q)$. Thus
\begin{eqnarray*}
\Vol\left(\varphi^T(\dot D^*_{q}Q)\right)&\geq&\int_{Q^0_{\rm gen}\cap Q^1_{\rm gen}}\dim\RFC_+^T(\varphi^t;T_q^*Q,T_{q'}^*Q) \;dq'\\
&\geq&\int_{Q^0_{\rm gen}\cap Q^1_{\rm gen}}\dim(\iota_{+}^{T,\infty})_*\RFH_+^{T}(\varphi^t;T_q^*Q,T_{q'}^*Q) \;dq'.
\end{eqnarray*}
Denote by $Q^2_{\rm gen}$ the set of points $q'\in Q$ that are not conjugate to $q$ with respect to the geodesic flow $\varphi^t_g$ of $g$. By Theorem~\ref{continuation} for $q'\in Q^0_{\rm gen}\cap Q^1_{\rm gen}\cap Q^2_{\rm gen}$ the growth rate of $\dim(\iota_{+}^{T,\infty})_*\RFH_+^{T}(\varphi^t;T_q^*Q,T_{q'}^*Q)$ is the same as the growth rate of $\dim(\iota_{+}^{T,\infty})_*\RFH_+^{T}(\varphi^t_g;T_q^*Q,T_{q'}^*Q)$. The group $(\iota_{+}^{T,\infty})_*\RFH_+^T(\varphi_g^t;T_q^*Q,T_{q'}^*Q)$ is isomorphic to the group $(\iota^{T,\infty})_*\HM^T(\Ecal;\Omega Q(q,q'))$ by Theorem~\ref{merry}. If $\gamma_{\rm exp}(\Omega Q_0(q))>0$ or $\gamma_{\rm exp}(\pi_1(Q))>0$, then 
$$\liminf_{T\to\infty}\frac 1T\log\dim (\iota^{T,\infty})_*\HM^T(\mathcal E; \Omega Q(q,q'))>0$$
 by Theorem~\ref{theoremheap}. Altogether, $\gamma_{\rm{vol,exp}}(\varphi^t;S^*_qQ)>0$. If both $\gamma_{\rm pol}(\Omega Q_0(q))$ and $\gamma_{\rm pol}(\pi_1(Q))$ are finite, then 
$$\liminf_{T\to\infty}\frac 1{\log T}\log\dim(\iota^{T,\infty})_* \HM^T(\mathcal E; \Omega Q(q,q'))\geq\gamma_{\rm pol}(\pi_1(Q)) + \gamma_{\rm pol}(\Omega Q_0(q))$$
 by Theorem~\ref{theoremheap}. Altogether, $\gamma_{\rm{vol,pol}}(\varphi^t;S^*_qQ)\geq\gamma_{\rm{pol}}(\Omega Q_0(q))+\gamma_{\rm{pol}}(\pi_1)-1$.

\section{Proof of Theorem~\ref{continuation}}\label{sec:appendix}
 The Rabinowitz--Floer homology we used in Section~\ref{sec:proof} was constructed by Albers--Frauenfelder in~\cite{AF1}. For the proof of Theorem~\ref{continuation}, they do not give details.
In this paper we use a sandwiching argument which allows us to concentrate on monotone deformations, bypassing the problems that arise for more general deformations. We first introduce the action functional properly and then prove Theorem~\ref{continuation}.\\

To define the action functional~(\ref{functional}), we want to associate to a positive contactomorphism $\varphi:S^*Q\to S^*Q$ a Hamiltonian on $T^*Q$. First we choose a positive path $\{\varphi^t\}_{t\in[0,1]}$ with $\varphi^0=id$ and $\varphi^1=\varphi$. We represent the spherization as $(\Sigma\subset T^*Q,\lambda|_\Sigma)$ as in Section~\ref{sec:tce} and generate $\{\varphi^t\}_{t\in[0,1]}$ by the contact Hamiltonian $h^t:\Sigma\times [0,1]\to\RR$. Since the path is positive, $h^t>0$. As explained in Section~\ref{sec:tce} we can choose $h^t=c_0$ for some fixed constant $c_0$ in a neighbourhood of $0$ and $1$. We extend $h^t$ on $\Sigma\times\RR$ constantly for $t\leq 0$ and 1-periodically for $t\geq0$.\\ 

Embed the symplectization of $\Sigma$ in $T^*Q$ and extend the coordinate $r$ by $r=0$ on $T^*Q\backslash(\Sigma\times\RR_{>0})$. Fix $\kappa, R\geq 1$, 
$$m\ \leq\ \min\{h^t(\theta)\mid\theta\in S^*Q,\, t\in S^1\},\quad M\ \geq\ \max\{h^t(\theta)\mid\theta\in S^*Q,\, t\in S^1\}$$
 and choose smooth functions $\beta:\RR_{\geq0}\to[0,1]$ and $h_{m}^M:\RR_{\geq0}\to \RR_{>0}$ with 
\begin{eqnarray*}
	\beta(r)&=&\begin{cases}0&\quad\mbox{if}\quad r\leq1 \quad \mbox{or}\quad r\geq R\kappa+1,\\1&\quad\mbox{if}\quad 2\leq r\leq R\kappa,\end{cases}\\
	h_m^M(r) &=&\begin{cases}m&\quad\mbox{if}\quad r\leq 2,\\M&\quad\mbox{if}\quad R\kappa\leq r.\end{cases}
\end{eqnarray*}
Then we define the Hamiltonian
\begin{eqnarray*}\label{Hamiltonian}
	H^t_{\kappa,R}(\theta,r)&=&r\Big(\beta(r)\,h^t(\theta)+(1-\beta(r))\,h_m^M(r)\Big)-\kappa.
\end{eqnarray*}
Apart from the shift by $-\kappa$, the Hamiltonian $H^t_{\kappa,R}$ is the 1-homogenous extension of $h^t$ for $2\leq r\leq R\kappa$ and the Hamiltonian of a Reeb flow for $r\leq 1$ and $r\geq R\kappa+1$. The Rabinowitz--Floer functional $\Acal_{\kappa,R}:=\Acal_{\kappa,R}(\varphi^t;q,q'):\Omega^1_{q,q'}\times\RR\to\RR$ is given by
\begin{eqnarray*}\label{functional1}
\Acal_{\kappa,R}(\varphi^t;q,q')(x,\eta)&=&\frac1\kappa\left(\int_0^1x^*\lambda-\eta\int_0^1H^{\eta t}_{\kappa,R}(x(t))\,dt\right).
\end{eqnarray*}
A pair $(x,\eta)\in\Omega^1_{q,q'}\times\RR$ is a critical point of $\Acal_{\kappa,R}$ if and only if it satisfies the equations
\begin{eqnarray*}\label{crit}
\begin{cases}
		\dot x(t)&=\:\: \eta X_{H_{\kappa,R}^{\eta t}}(x(t)),\\
		H_{\kappa,R}^{\eta}(x(1))&=\:\: 0.
\end{cases}
\end{eqnarray*}
 The factor $\frac1\kappa$ does not change the critical points, just the critical values. The functional $\Acal_{\kappa,R}$ depends on the choice of $\kappa, R$, but the following lemma shows that for large enough $\kappa, R$ the critical points with action in a fixed range are independent of the choice. This justifies that we suppressed $\kappa$ and $R$ in the main text. 

\begin{lemma}\label{independence}
	Given $a<b$, there are constants $\kappa_0,R_0$ such that for $\kappa\geq\kappa_0$ and $R\geq R_0$ the following holds. If $(x,\eta)$ is a critical point with $a\leq\Acal_{\kappa,R}(x,\eta)\leq b$, then the radial component of $x$ stays in $[2,R\kappa]$ for $t\in[0,1]$ and $\Acal_{\kappa,R}(x,\eta)=\eta$. 
\end{lemma}
\begin{proof} A detailed proof is given in~\cite[Proposition 4.3]{AF1}. \end{proof}

We define the $L^2$-metric $g^\kappa$ on $\Omega_{q,q'}^1\times\RR$ by choosing an almost complex structure $J$ compatible with $\omega$ and setting
\begin{eqnarray*}
	g^\kappa((\hat x,\hat\eta),(\hat x',\hat \eta'))&=&\frac1\kappa\int_0^1\omega(\hat x,J\hat x')\,dt+\frac{\hat\eta\hat\eta'}\kappa.
\end{eqnarray*} 
With this scalar product the gradient of $\Acal_{\kappa,R}(\varphi^t;q,q')$ has the form
\begin{eqnarray*}\label{gradient}
	\nabla \Acal_{\kappa,R}(\varphi^t;q,q')(x,\eta) &=& \left(
	\begin{array}{c}
		\dot x-\eta X_{H^{\eta t}}(x(t))\\
		\int_0^1 H^{\eta t}(x(t))+\eta t\dot H^{\eta t}(x(t))\;dt
	\end{array}
		 \right).
\end{eqnarray*}

Assume that the functional $\Acal_{\kappa,R}(\varphi^t;q,q')$ is Morse. The chain complex $\RFC^b=\RFC^b(\varphi^t;q,q')$ of the filtered Rabinowitz--Floer homology is generated by the critical points of $\Acal_{\kappa,R}(\varphi^t;q,q')$ with action $\leq b\in\RR$. The boundary operator $\partial^b$ is defined by counting solutions of the negative gradient flow. For $a\leq b$ denote by $\iota^{a,b}:\RFC^a\to\RFC^{b}$ the inclusion. Choose $\kappa_0, R_0$ so large that Lemma~\ref{independence} holds for critical points with action in~$[a,b]$. Denote $\RFC_a^b=\RFC^b/\iota^{a,b}(\RFC^a)$ and set $\RFH_a^b=H(\RFC_a^b,\partial_a^b)$, where $\partial_a^b:\RFC_a^b\to\RFC_a^b$ is the induced boundary operator. These groups are independent of $\kappa\geq\kappa_0,R\geq R_0$. For such $\kappa,R$ and for $a'\leq a\leq b\leq b'$ let $\iota_{a,a'}^{b,b'}:\RFC_a^b\to\RFC_{a'}^{b'}$ be the homomorphism induced by inclusions. Denote by $\RFC_{-\infty}^b=\lim_{a\to-\infty}\RFC_a^b$ the inverse limit and for $a\in\RR\cup\{-\infty\}$ by $\RFC_a^\infty=\lim_{b\to\infty}\RFC_a^b$ the direct limit, while adjusting $\kappa,R$. For $\kappa\geq\kappa_0,R\geq R_0$, $\RFC_a^b=\RFC^b_{-\infty}/\iota_{-\infty,-\infty}^{a,b}(\RFC_{-\infty}^a)$. For better readability we omit the subscript $-\infty$ and denote by $\RFC_+^T=\RFC_0^T$ the positive part of the chain complex, by $\iota_+^{T,T'}=\iota_{0,0}^{T,T'}$ the inclusion and by $\RFH_+^T=\RFH_0^T$ the positive part of the homology.\\  

Consider now a family $\varphi_s^{t}$ of paths of contactomorphisms induced by a family of Hamiltonians $h_s^t$ such that $\partial_sh_s^t=0$ for $s\notin[0,1]$. Suppose that for the associated family of functionals $\Acal_s(\varphi_s^{t}):=\Acal_{\kappa,R}(\varphi^t_s;q,q')$ the constants from Lemma~\ref{independence} are chosen uniformly large enough. We set $\Acal_-=\Acal_s$ for $s\leq0$ and $\Acal_+=\Acal_s$ for $s\geq1$. The continuation homomorphism $\Phi:\RFC^\infty(\Acal_-)\to\RFC^\infty(\Acal_+)$ is defined in the standard way by counting solutions $(x(s),\eta(s))$ of the equation
\begin{eqnarray}\label{conteqn}
	\partial_s(x(s),\eta(s))&=&-\nabla\Acal_s(x(s),\eta(s)),
\end{eqnarray}
such that $\lim_{s\to\pm\infty}(x_s,\eta_s)=(x_\pm,\eta_\pm)$ exist and are critical points of $\Acal_\pm$. Then $\Phi$ induces an isomorphism $\RFH^\infty(\Acal_-)\to \RFH^\infty(\Acal_+)$, because $\eta$ is bounded along deformations, cf.~\cite[Corollary 3.4]{CieFra09}.

\begin{proof}[Proof of Theorem~\ref{continuation}] 

First we consider two positive Hamiltonians $h^t_0,h^t_1$ such that $h^t_0\leq h^t_1$. We show that the action is non-increasing along solutions of~(\ref{conteqn}).\\

For the deformation from $h_0^t$ to $h_1^t$ define a function $\chi:\RR\to[0,1]$ such that $\chi'\geq 0$ and
\begin{eqnarray*}
	\chi(s)&=&\begin{cases}0&\quad\mbox{if}\quad s\leq0,\\
	1&\quad\mbox{if}\quad s\geq 1,\end{cases}
\end{eqnarray*}
and set $h_s^t:=h^t_0 + \chi(s)(h_1^t-h_0^t)$. Denote by $H_s^t$, $\varphi_s^t$ and $\Acal_s$ the associated Hamiltonians, paths of contactomorphisms and functionals. The deformation satisfies 
$$\begin{array}{ccccccc}
\frac{d}{ds}H_s^t &=& \chi'(s)(H_1^t-H_0^t) &=& \chi'(s)\,r\,\beta(r)(h^t_1-h_0^t) &\geq& 0.
\end{array}$$

 For $(x,\eta)\in\Omega_{q,q'}^1\times\RR$,
\begin{eqnarray*}
	\frac{\partial}{\partial s} \Acal_s(x,\eta)&=&\int_0^1-\frac\eta\kappa\chi'(s)(H_1^{\eta t}-H_0^{\eta t}) \;dt.\nonumber
\end{eqnarray*}

 Now consider a solution $(u(s),\eta(s))$ of~(\ref{conteqn}). Denote $E=\int_{-\infty}^\infty \|\partial_s (x(s),\eta(s))\|^2\;ds$ and $\Acal_\pm=\Acal_\pm(u_\pm,\eta_\pm)$. 
We calculate
\begin{eqnarray}
	\Acal_+&=&\Acal_-+\int_{-\infty}^\infty\frac d{ds}\Acal_s(x(s),\eta(s))\;ds\nonumber\\
	&=&\Acal_-+\int_{-\infty}^\infty\Big(\frac\partial{\partial s}\Acal_s\Big)(x(s),\eta(s))+\big\langle\nabla\Acal_s\big(x(s),\eta(s)\big),\partial_s\big(x(s),\eta(s)\big)\big\rangle \;ds\nonumber\\
	&=&\Acal_--E+\int_{-\infty}^\infty\int_0^1-\frac{\eta(s)}\kappa\chi'(s)(H_1^{\eta t}-H_0^{\eta t}) \;dt\;ds\nonumber
\end{eqnarray}
 If $\eta(s)\geq0$, then $-\frac{\eta(s)}\kappa\chi'(s)(H_1^{\eta t}-H_0^{\eta t})\leq0$. If $\eta(s)\leq0$, then $h_0^{\eta t}=h_1^{\eta t}=c_0$ and thus $-\frac{\eta(s)}\kappa\chi'(s)(H_1^{\eta t}-H_0^{\eta t})=0$. It follows that $\Acal_+\leq\Acal_-$. \\

We have just shown that $\Phi$ restricts to $\Phi|_{\RFC^T(\varphi_0^t)}:
\RFC^T(\varphi_0^t)\to \RFC^T(\varphi_1^t)$. Furthermore $\varphi_0^t=\varphi_1^t$ for $t\leq 0$. Thus $\Acal(\varphi_0^t)$ and $\Acal(\varphi_1^t)$ have the same critical points with nonpositive action, and constant critical points $(x,\eta)$ with $\eta\leq0$ are solutions of~(\ref{conteqn}). Together with the fact that the action is non-increasing along solutions of~(\ref{conteqn}) we get that 
$$\Phi|_{\RFC^0(\varphi_0^t)}\;:\;\RFC^0(\varphi_0^t)\;\to\; \RFC^0(\varphi_1^t)$$
 is a lower diagonal isomorphism. Thus for the homomorphism $\Phi_*$ induced in the quotient we have
\begin{eqnarray*}
	\Phi_*(\RFC_+^T(\varphi_0^t))&=&\Phi(\RFC^T(\varphi_0^t))/\Phi(\RFC^0(\varphi_0^t))\\
	&=&\Phi(\RFC^T(\varphi_0^t))/\iota^{0,T}\RFC^0(\varphi_1^t)\\
	&\subseteq&\RFC_+^T(\varphi_1^t).
\end{eqnarray*}

Since $\Phi$ induces an isomorphism in $\RFH^\infty$, abbreviating $\iota=\iota_{+}^{T,\infty}$, we conclude that 
\begin{eqnarray}\label{inequality}
	\dim \iota_*\RFH_+^T(\varphi_0^t)&\leq&\dim \iota_*\RFH_+^T(\varphi_1^t).
\end{eqnarray}

Now choose $c,C>0$ such that $c\leq h^t\leq C$. Denote by $\varphi_c^t,\varphi_{h^t}^t$ and $\varphi_C^t$ the induced flows. The constants $c,C$ are not equal to $c_0$ for $t$ near $0$ or $1$, so we need to modify them to fit our setup. From the proof of~\cite[Proposition~6.2]{FLS} it becomes clear that there are functions $h_c^t,h_C^t:S^*Q\times[0,1]\to\RR$ with $h_c^t=h_C^t=c_0$ for $t$ near 0 and 1, that satisfy $h_c^t\leq h^t\leq h_C^t$, and such that the flows $\varphi_{h_c^t}^t$ and $\varphi_{h_C^t}^t$ induced by $h_c^t$ and $h_C^t$ are time-reparametrizations of the geodesic flows $\varphi_c^t$ and $\varphi_C^t$ that satisfy $\varphi^1_{h_c^t}=\varphi^1_c$ and $\varphi^1_{h_C^t}=\varphi^1_C$. Extend $h_c^t$ and $h_C^t$ as in Section~\ref{sec:tce}, see Figure~\ref{fig:graph2}. 

\begin{figure}[h]
	\centering
		\includegraphics{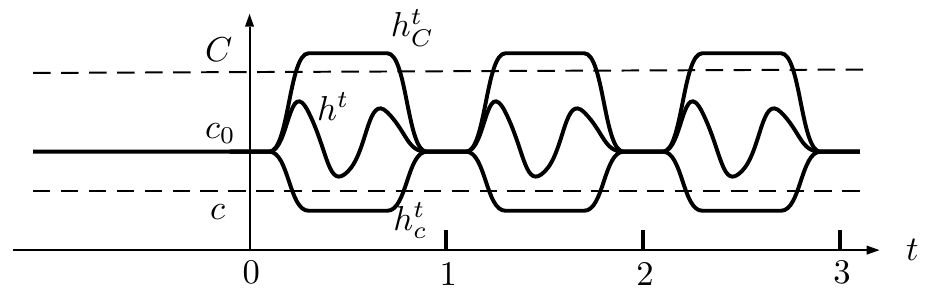}
	\caption{The functions $h_c^t$, $h^t$ and $h_C^t$, extended to $\RR$}.
	\label{fig:graph2}
\end{figure}

We apply~(\ref{inequality}) twice, first to a monotone deformation from $h_c^t$ to $h^t$ and then to a monotone deformation from $h^t$ to $h_C^t$. 
By construction of $h_c^t$ and $h_C^t$ it is clear that there exists a function $\tau:\RR_{\geq0}\to\RR_{\geq0}$ such that $\varphi_{h_c^t}^{\tau(t)}=\varphi_{h_C^t}^t$ and such that $t\leq\tau(t)\leq2\frac Cc t$. Thus,
\begin{eqnarray*}
	\iota_*\RFH_+^{\tau(T)}(\varphi_{h_c^t}^t)&\cong&\iota_*\RFH_+^T(\varphi_{h_C^t}^t).
\end{eqnarray*}
With~(\ref{inequality}), this results in 
$$\dim \iota_*\RFH_+^T(\varphi_{h_c^t}^t)\leq\dim \iota_*\RFH_+^T(\varphi^t)\leq\dim \iota_*\RFH_+^T(\varphi_{h_C^t}^t)=\dim \iota_*\RFH_+^{\tau(T)}(\varphi_{h_c^t}^t).$$
We conclude that for every positive path of contactomorphisms the Rabinowitz--Floer homology grows as fast as for a Reeb flow.
\end{proof}

\section{Complementary examples}\label{examples}

In this section we provide two examples that confirm Example~\ref{prop:examples}.\\
\\
The first example is described in detail in~\cite[Section~7]{MacSch11}. Consider the semidirect product $G=\RR^2\ltimes\RR$ with multiplication 
$$(x,y,z)\bullet(x',y',z')=(x+e^zx',y+e^{-z}y',z+z').$$
Choose a cocompact lattice $\Gamma\subseteq G$ of exponential growth. Then the fundamental group $\Gamma$ of the closed manifold $Q=\Gamma\backslash G$ has exponential growth. The functions 
$$(M_x,M_y,M_z)=(e^zp_x,e^{-z}p_y,p_z)$$
on $T^*G$ are left-invariant and thus descend to $T^*Q$. Consider the Hamiltonian
$$H=\frac12\left((M_x+1)^2+M_y^2+M_z^2\right)$$
that describes an exact magnetic flow on $Q$. For every energy value $k>0$ the hypersurface $\Sigma_k=H^{-1}(k)$ is contactomorphic to the spherization $S^*Q$. The zero section is enclosed by $\Sigma_k$ if and only if $k>\frac 12$, and the zero section is contained in $\Sigma_{\frac 12}$. For every $k>\frac 12$ the flow $\varphi_k^t$ induced by $H$ on $\Sigma_k$ is a positive contact isotopy and thus has positive topological entropy. On the other hand Macarini--Schlenk showed that for $k\leq\frac12$ the flow $\varphi_k^t$ has zero topological entropy. In particular, $\varphi_{\frac12}^t$ is the smooth limit of positive contact isotopies, and thus a non-negative contact isotopy. It fails to be positive only on the zero section, which is a codimension $2$ subset of $\Sigma_{\frac12}$.

The second example was pointed out to me by Marcelo Alves. 
Let $\Sigma_k$ be the closed orientable surface of genus $k\geq2$. Choose a Riemannian metric $g$ on $\Sigma_k$ such that a closed disc $D\subseteq\Sigma_k$ is isometric to a round sphere $S^2$ deprived of an open disc that is strictly contained in a hemisphere. Equip the spherization $S^*\Sigma_k$ with the contact form whose Reeb flow is the geodesic flow $\varphi_g^t$. Let $U\subseteq S^*\Sigma_k$ be the set of points whose $\varphi_g^t$ flow lines intersect fibers over $\Sigma_k\backslash D$. By construction, $U^c$ is closed with non-empty interior. Further, the geodesic flow $\varphi_g$ on $U^c$ is periodic (the flow lines project to great circles on $S^2$). Thus $U^c$ is a closed invariant set on which $\varphi_g^t$ has zero topological entropy, and $\overline U$ is also a closed invariant set. From the maximum formula for topological entropy on decompositions into closed invariant sets~\cite[Proposition 3.1.7(2)]{HaKa} we conclude that $\varphi_g$ has positive topological entropy on $\overline U$. Now consider the contact Hamiltonian flow $\varphi^t$ induced by a Hamiltonian that is constant $1$ on $\overline U$ and $0$ outside a small neighbourhood $\widetilde U$ of $\overline U$. Since $\varphi^t$ coincides with $\varphi_g^t$ on $\overline U$, it has positive topological entropy, but restricts to the identity on $\widetilde U^c$, which is a set with nonempty interior. Note that for all $\varepsilon>0$ we can choose $g$ such that $\mu_g(\widetilde U^c)\geq (1-\varepsilon)\mu_g(S^*(\Sigma_k))$, where $\mu_g$ is the measure induced by $g$.


\end{document}